\documentclass[12pt]{article}
\usepackage{amsmath,amssymb,amsthm,amsfonts,amsbsy, enumerate}
\usepackage[dvips]{epsfig}

\theoremstyle{plain}
\newtheorem{theorem}{Theorem}[section]
\newtheorem{corollary}[theorem]{Corollary}
\newtheorem{conj}[theorem]{Conjecture}
\newtheorem{lemma}[theorem]{Lemma}
\newtheorem{prop}[theorem]{Proposition}

\def\theta{\vartheta}

\begin{document}

\title{The chromatic index of strongly regular graphs}
\author{Sebastian M. Cioab\u{a}\thanks{Department of Mathematical Sciences,
University of Delaware, Newark, Delaware 19716-2553, USA, {\tt cioaba@udel.edu}.
Research supported by NSF grants DMS-1600768 and CIF-1815922.} \and
Krystal Guo\thanks{Korteweg-De Vries institute, University of Amsterdam, Amsterdam, The Netherlands, {\tt  k.guo@uva.nl}.
This research was done while K. Guo was a postdoctoral fellow at Universit\'{e} Libre de Bruxelles, supported by ERC Consolidator Grant 615640-ForEFront.} \and
Willem H. Haemers\thanks{Department of Econometrics and Operations Research, Tilburg University, Tilburg, The Netherlands,
{\tt haemers@uvt.nl}.}
}
\date{}
\maketitle

\begin{abstract}
We determine (partly by computer search) the chromatic index (edge-chromatic number) of many
strongly regular graphs (SRGs), including the SRGs of degree $k\leq 18$ and their complements,
the Latin square graphs and their complements, and the triangular graphs and their complements.
Moreover, using a recent result of Ferber and Jain, we prove that an SRG of even order $n$,
which is not the block graph of a Steiner $2$-design or its complement, has chromatic index $k$, when $n$ is big enough.
Except for the Petersen graph, all investigated connected SRGs of even order have chromatic index equal to $k$,
i.e., they are class $1$, and we conjecture that this is the case for all connected SRGs of even order.
\\[5pt]
\end{abstract}

\section{Introduction}

An {\em edge-coloring} of a graph $G$ is a coloring of its edges such that intersecting edges
have different colors.
Thus a set of edges with the same colors (called a color class) is a matching.
The {\em edge-chromatic number} $\chi'(G)$ (also known as the {\em chromatic index}) of $G$ is the minimum
number of colors in an edge-coloring.
By Vizing's famous theorem \cite{Vi}, the chromatic index of a graph $G$ of maximum degree $\Delta$ is $\Delta$ or $\Delta+1$.
A graph with maximum degree $\Delta$ is called class~1 if $\chi'(G)=\Delta$ and is called class~$2$ if $\chi'(G)=\Delta+1$.
It is also known that determining whether a graph $G$ is class~$1$ is an NP-complete problem \cite{Holyer}.
If $G$ is regular of degree $k$, then $G$ is class~1 if and only if $G$ has an edge
coloring such that each color class is a perfect matching.
A perfect matching is also called a {\em $1$-factor}, and a partition of the edge set into perfect matchings is called a
{\em $1$-factorization}.
So being regular and class~1 is the same as having a $1$-factorization (being $1$-factorable), and requires that the graph
has even order.

A graph $G$ is called a {\em strongly regular graph} (SRG) with parameters $(n,k,\lambda,\mu)$
if it has $n$ vertices, is $k$-regular ($0<k<n-1$), any two adjacent vertices of $G$ have exactly $\lambda$ common neighbors and any two
distinct non-adjacent vertices of $G$ have exactly $\mu$ common neighbors.
The complement of a strongly regular graph with parameters $(n,k,\lambda,\mu)$ is again
strongly regular, and has parameters $(n,n-k-1,n-2k+\mu-2,n-2k+\lambda)$.
An SRG $G$ is called {\em imprimitive} if $G$ or its complement is disconnected, and {\em primitive} otherwise.
An imprimitive SRG must be $\ell K_m$ ($\ell,m\geq 2$), the disjoint union of $\ell$ cliques of order $m$,
or its complement.
It is well-known that $K_m$ ($m\geq 2$), and hence also $\ell K_m$, is class~1 if and only if $m$ is even.
The complement of $\ell K_m$ is a regular complete multipartite graph which is known to be class~1 if and only if the order is
even \cite{HR}.

A {\em vertex coloring} of $G$ is a coloring of the vertices of $G$ such that adjacent vertices have different colors.
The {\em chromatic number} $\chi(G)$ of $G$ is the minimum number of colors in a vertex coloring.
For the chromatic number there exist bounds in terms of the eigenvalues of the adjacency matrix,
which turn out to be especially useful for strongly regular graphs (see for example \cite{BHBook}).
These bounds imply that there exist only finitely many primitive SRGs with a given chromatic number,
and made it possible to determine all SRGs with chromatic number at most four (see~\cite{H}).
Motivated by these results, Alex Ros\'{a} asked the third author whether eigenvalue techniques can give
information on the chromatic index of an SRG.
There exist useful spectral conditions for the existence of a perfect matching (see~\cite{BH,CGH}), and
Brouwer and Haemers \cite{BH} have shown that every regular graph of even order, degree $k$ and second largest
eigenvalue $\theta_2$ contains at least 
$\lfloor (k-\theta_2 +1)/2 \rfloor$ edge disjoint perfect matchings.
From this it follows that every connected SRG of even order has a perfect matching. Moreover,
Cioab\u{a} and Li \cite{CL} proved that any matching of order $k/4$ of a primitive SRG of valency $k$
and even order, is contained in a perfect matching. These authors conjectured that $k/4$ can be
replaced by $\lceil k/2\rceil-1$ which would be best possible. Unfortunately, we found no useful
eigenvalue tools for determining the chromatic index. However, the following recent result of Ferber and
Jain~\cite{FJ} gives an asymptotic condition for being class~1 in terms of the eigenvalues.
\begin{theorem}\label{fj}
There exist universal constants $n_0$ and $k_0$, such that the following holds.
If $G$ is a connected $k$-regular graph of even order $n$ with eigenvalues $k=\theta_1>\theta_2\geq\ldots\geq\theta_n$,
and $n>n_0$, $k>k_0$ and $\max\{\theta_2,-\theta_n\}<k^{0.9}$, then $G$ is class~$1$.
\end{theorem}
If $G$ has diameter~$2$ (as is the case for a connected SRG), then $n\leq k^2+1$.
This implies that for an SRG we do not need to require that $k>k_0$ when we take $n_0\geq k_0^2+1$.
Theorem~\ref{fj} enables us to show that, except for one family of SRGs, all connected SRGs of even order $n$ are class~1,
provided $n$ is large enough.
In addition, we present a number of sufficient conditions for an SRG to be class~1.
By computer, using SageMath \cite{sage}, we verified that all primitive SRGs of even order and degree $k\leq 18$ and their complements
are class~1,
except for the Petersen graph, which has parameters $(10,3,0,1)$ and edge-chromatic number $4$ (see \cite{NS,V} for example).
We also determine the chromatic index of several other primitive SRGs of even order, and all are class~1.
Therefore we believe:
\begin{conj}
Except for the Petersen graph, every connected SRG of even order is class~1.
\end{conj}

\section{Sufficient conditions for being class~1}\label{suf}

A well known conjecture (first stated by Chetwynd and Hilton \cite{ChWi}, but attributed to Dirac) states that every $k$-regular graph
of even order $n$ with $k\geq n/2$ is $1$-factorable.
Cariolaro and Hilton~\cite{CH} proved that the conclusion holds when $k\geq 0.823 n$,
and Csaba, K\"uhn, Lo, Osthus, and Treglown~\cite{CKLOT}, proved the following result.
\begin{theorem}\label{cklot}
There exists a universal constant $n_0$, such that if $n$ is even, $n>n_0$ and if $k\geq 2\lceil n/4\rceil -1$,
then every $k$-regular graph of order $n$ has chromatic index $k$.
\end{theorem}
K\"onig~\cite{K} proved that every regular bipartite graph of positive degree has a $1$-factorization.
We need the following generalization of K\"onig's result.
\begin{lemma}\label{ck}
Let $G=(V,E)$ be a connected regular graph of even order $n$,
and let $\{V_1,V_2\}$ be a partition of $V$ such that $|V_1|=|V_2|=n/2$.
\\
(i) If the graphs induced by $V_1$ and $V_2$ are $1$-factorable, then so is $G$.
\\
(ii) If $V_1$ (and hence $V_2$) is a clique or a coclique, then $G$ is class~1.
\end{lemma}
\begin{proof}
Partition the edge set $E$ into two classes $E_1$ and $E_2$, where $E_1$ contains all edges with both endpoints in the same vertex set
$V_1$ or $V_2$, and the edges of $E_2$ have one endpoint in $V_1$ and the other endpoint in $V_2$.
\\
(i)~If the graphs induced by $V_1$ and $V_2$ are $1$-factorable, then both have the same degree, and therefore also $(V,E_1)$ is
$1$-factorable.
By K\"onig's theorem $(V,E_2)$ is $1$-factorable, therefore $G$ is class~1.
\\
(ii)~If $V_1$ is a coclique, then so is $V_2$ and we have the theorem of K\"onig.
If $V_1$ is a clique, then so is $V_2$.
If $n/2$ is even, then the result is proved in (i).
If $n/2$ is odd,
then we choose a 1-factor $F$ in $(V,E_2)$ (here we use that $G$ is connected),
and define $E_2'=E_2\setminus F$ and $E_1'=E_1\cup F$.
Then $(V,E'_2)$ is $1$-factorable (or has no edges), and $(V,E'_1)$ consists of two cliques of order $n/2$ and the $1$-factor $F$.
Thus $F$ gives a bijection between $V_1$ and $V_2$.
We color the edges of both cliques with $n/2$ colors, such that the bijection $F$ preserves the edge colors.
Now for each edge $e$ of $F$, the two sets of colored edges that intersect at an endpoint of $e$ use the same set of $n/2-1$ colors.
So we can color $e$ with the remaining color.
\end{proof}
There exist several SRGs that have the partition of case~(i).
The Gewirtz graph is the unique SRG with parameters $(56,10,0,2)$,
and admits a partition into two Coxeter graphs (see~\cite{BHGewirtz}).
The Coxeter graph is known to be $1$-factorable (see~\cite{B'}), therefore the Gewirtz graph is class~1.
The same holds for the point graph of the generalized quadrangle $GQ(3,9)$ (the unique SRG(112,30,2,10)),
which admits a partition into two Gewirtz graphs, and for the Higman-Sims graph (the unique SRG(100,22,0,6)),
which can be partitioned into two copies of the Hoffman-Singleton graph (the unique strongly regular graph with parameters (50,7,0,1)),
which has chromatic index~7 (see Section~\ref{comp}).

Suppose $\theta_1\geq \theta_2\geq \dots \geq \theta_n$ are the eigenvalues of a graph $G$ of order $n$.
Hoffman (see \cite[Theorem 3.6.2]{BHBook} for example) proved that the chromatic number of $G$ is at least $1-\theta_1/\theta_n$. 
A vertex coloring that meets this bound is called a {\em Hoffman coloring}.
For $k$-regular graphs, the color classes of a Hoffman coloring are cocliques of which the size meets Hoffman's coclique bound
$n\theta_n/(\theta_n-k)$.
This implies (see~\cite{BHBook} for example) that all the color classes have equal size,
and any vertex $v$ of $G$ has exactly $-\theta_n$ neighbors in each color class different from the color class of $v$.
%
\begin{theorem}
Suppose $G=(V,E)$ is a $k$-regular graph with an even chromatic number that meets Hoffman's bound.
Then both $G$ and its complement $\overline{G}$ are class~$1$,
or $\overline{G}$ is a disjoint union of cliques of odd order.
\end{theorem}
\begin{proof}
Let $S_1,\dots,S_{2t}$ be the color classes in a Hoffman coloring of $G$.
This implies that each $S_i$ is a coclique attaining equality in the Hoffman ratio bound,
which means that each vertex outside $S_i$ has exactly $-\theta_n$ neighbors in $S_i$.
Hence, each subgraph induced by two distinct cocliques $S_i$ and $S_j$ is a bipartite regular graph of valency $-\theta_n$.
A $1$-factorization of $K_{2t}$ corresponds to a partition $E_1,\ldots,E_{2t-1}$ of $E$, such that each $(V,E_i)$ consists of $t$ disjoint
regular bipartite graphs of degree $-\theta_n=k/(2t-1)$.
By K\"onig's theorem it follows that each $(V,E_i)$ is $1$-factorable, and therefore $G$ is class~1.
\\
For the complement $\overline{G}=(V,F)$ of $G$, a similar approach works.
We can partition the edge set $F$ into the subsets $F_0,F_1,\ldots,F_{2t-1}$, such that $(V,F_i)$ is the disjoint union of
$t$ regular bipartite graphs of the same positive degree for $i=1,\ldots,2t-1$
(if the degree is $0$, then $\overline{G}$ is a disjoint union of cliques).
But now there is an additional graph $(V,F_0)$ consisting of $2t$ disjoint cliques.
We combine $F_0$ and $F_1$.
Then $(V,F_0 \cup F_1)$ is the disjoint union of $t$ complements of regular incomplete bipartite graphs with the same positive degree,
and therefore has a $1$-factorization by Lemma~\ref{ck}.
Since $(V,F_i)$ has a $1$-factorization for $i=2,\ldots,2t-1$, it follows that $\overline{G}$ is $1$-factorable.
\end{proof}
For an SRG the color partition of a Hoffman coloring corresponds to a so-called {\em spread} in the complement (see~\cite{HT}).
As a consequence of this result, it follows that any primitive strongly regular graph with a spread with an even number of cliques,
or a Hoffman coloring with an even number of colors is class~1.
Among such SRGs are the Latin square graphs.
Consider a set of $t$ ($t\geq 0$) mutually orthogonal Latin squares of order $m$ ($m\geq 2$).
The vertices of the Latin square graph
are the $m^2$ entries of the Latin squares, and two distinct entries are adjacent if they lie in the same row, the same column,
or have the same symbol in one of the squares.
If $t=m-1$ we obtain the complete graph $K_{m^2}$, and if $t=m-2$ we have a complete multipartite graph.
Otherwise the Latin square graph is a primitive SRG with parameters $(m^2, (t+2)(m-1), m-2+t(t+1), (t+1)(t+2))$.
If $t=0$ we only have the rows and columns, then the Latin square graph is better known as the Lattice graph $L(m)$.
If $m\neq 4$, the Lattice graph is determined by the parameters.
The $m$ rows of a Latin square give a partition of the vertex set of the Latin square graph into cliques, which is a spread.
Thus we have:
\begin{corollary}\label{ls}
If $G$ is a Latin square graph of even order, then both $G$ and its complement are $1$-factorable.
\end{corollary}

\section{Asymptotic results}

A {\em Steiner $2$-design} (or {\em $2$-$(m,\ell,1)$ design}) consists of a point set $\cal P$ of cardinality $m$,
together with a collection of subsets of $\cal P$ of size $\ell$ ($\ell\geq 2$), called {\em blocks}, such that every pair of points
from $\cal P$ is contained in exactly one of the blocks.
The {\em block graph} of a Steiner 2-design is defined as follows.
The blocks are the vertices, and two vertices are adjacent if the blocks intersect in one point.
If $m=\ell^2-\ell+1$, the Steiner 2-design is a projective plane, and the block graph is $K_m$.
Otherwise the block graph is an SRG with parameters
$(m(m-1)/\ell(\ell-1), \ell(m-\ell)/(\ell-1), (\ell-1)^2+(m-2\ell+1)/(\ell-1), \ell^2)$.

\begin{theorem}\label{asympt}
There  exists an integer $n_0$, such that every primitive strongly regular graph of even order $n>n_0$,
which is not the block graph of a Steiner $2$-design or its complement, is class~$1$.
\end{theorem}

\begin{proof}
Suppose $G$ is a primitive $(n,k,\lambda,\mu)$-SRG of even order $n$.
Then it is well-known (see for example~\cite{BHBook}, Chapter~9) that $G$ has exactly three distinct eigenvalues
$k=\theta_1$, $\theta_2$ and $\theta_n$.
Moreover, the eigenvalues are nonzero integers and satisfy $k+\theta_2 \theta_n = \mu$.
Assume that $G$ nor its complement $\overline{G}$ is the block graph of a Steiner 2-design or a Latin square graph.
Using a result of Neumaier~\cite{N} (known as the claw bound), we get that \[ \theta_2\leq \theta_n(\theta_n+1)(\mu+1)/2-1. \]
Another result of Neumaier~\cite{N} (the $\mu$-bound) gives $\mu\leq \theta_n^3(2\theta_n+3)$.
Combining these inequalities, after some straightforward calculations, we obtain that $\theta_2 < (-\theta_n)^6$.
Since $k+\theta_2\theta_n=\mu>0$, we deduce that
\[
k^6 > (-\theta_2 \theta_n)^6 > \theta_2^6 \theta_2 = \theta_2^7, \mbox{ so } \theta_2<k^{6/7}.
\]
Next we apply the same result to $\overline{G}$, and obtain $-1-\theta_n < (1+\theta_2)^6$,
which yields $-\theta_n\leq (2\theta_2)^6$ (since $\theta_2$ is a positive integer). Hence
\[
k^6 > (-\theta_2 \theta_n)^6 > 2^{-6}(-\theta_n)(-\theta_n)^6  = 2^{-6}(-\theta_n)^7, \mbox{ so } -\theta_n < (2k)^{6/7}< k^{0.9},
\]
when $k$ is large enough.
Thus we get $\max\{\theta_2,-\theta_n\}\leq k^{0.9}$.
Now we apply the result of Ferber and Jain and conclude that $G$ is class~1 when $n$ is large enough.

If $G$ is a Latin square graph of even order then by Corollary~\ref{ls} both $G$ and its complement $\overline{G}$ are class~1.
\end{proof}

In many cases the complement of the block graph of a Steiner $2$-design has $k>n/2$, so it will have a $1$-factorization by
Theorem~\ref{cklot}, provided $n$ is even and large enough.
The following result follows straightforwardly from the mentioned result of Cariolaro and Hilton~\cite{CH}.
\begin{prop}\label{stcom}
If $G$ is the complement of the block graph of a $2$-$(m,\ell,1)$ design with $6\ell^2\leq m$,
then $G$ is class~1, provided $G$ has even order.
\end{prop}
For every $m\ge 2$ there is a unique $2$-$(m,2,1)$ design, and its block graphs is the triangular graph $T(m)$.
It is isomorphic to the line graph of the complete graph $K_m$, and if $m\ge 4$ $T(m)$ is an SRG with
parameters $(m(m-1)/2,2(m-2),m-2,4)$.
The triangular graph is uniquely determined by its parameters if $m\ne 8$.
Alspach~\cite{A} has proved that $T(m)$ has a $1$-factorization if the order is even, which is the case if $m\equiv 0$, or $1\!\!\mod 4$.
Proposition~\ref{stcom} implies that the complement of $T(m)$ is class~1 if $m\geq 24$ and the order is even.
The complement of $T(5)$ is the Petersen graph, which is class~2.
For $5<m<24$ and $m\equiv 0$, or $1\!\!\mod 4$ we found a 1-factorization in the complement of $T(m)$ by computer
(see next section for more about the computer search).
Thus we can conclude:
\begin{theorem}
For $m\equiv 0$, or $1\!\mod 4$ the triangular graph $T(m)$ is class~1, and if $m\neq 5$ so is its complement.
\end{theorem}
If the block size equals 3, the design is better known as a Steiner triple system.
The chromatic index of the block graph of a Steiner triple system is investigated in \cite{DPP}.
The paper contains several sufficient conditions for such a graph to be class~1,
and the authors conjecture that all these graphs are class~1 when the order is even.
One of the results from \cite{DPP} (Theorem~2.2) can be generalized to arbitrary Steiner 2-designs.
A set of $m/\ell$ disjoint blocks of a $2$-$(m,\ell,1)$ design is called a {\em parallel class},
and a partition of the block graph into parallel classes is a {\em parallelism}.
A parallelism of a Steiner $2$-design gives a Hoffman coloring in the block graph,
so we have:
\begin{prop}
If a Steiner $2$-design has a parallelism with an even number of parallel classes, then the block graph and its complement are class~1.
\end{prop}

\section{SRGs of degree at most $18$}\label{comp}

According to the list of Brouwer~\cite{B} all primitive SRGs of even order and degree at most 18 are known
(one only has to check the parameter sets up to $n=18^2+1=325$).
The parameters together with the number of nonisomorphic SRGs with $k<n/2$ are given in Table~\ref{SRG}
(the ones with $k\ge n/2$ are the complements of $a$ to $e$).
\begin{table}
\centering
\begin{tabular}{|c|c|c|}
\hline
\begin{tabular}{llr}
$a$&(10,3,0,1)&1\\
$b$&(16,5,0,2)&1\\
$c$&(16,6,2,2)&2\\
$d$&(26,10,3,4)&\!10\\
$e$&(28,12,6,4)&4
\end{tabular}
&
\begin{tabular}{llr}
$f$&(36,10,4,2)&1\\
$g$&(36,14,4,6)&180\\
$h$&(36,14,7,4)&1\\
$i$&(36,15,6,6)&\!32548\\
$j$&(40,12,2,4)&28
\end{tabular}
&
\begin{tabular}{llr}
$k$&(50,7,0,1)&1\\
$l$&(56,10,0,2)&1\\
$m$&(64,14,6,2)&1\\
$n$&(64,18,2,6)&\!167\\
$o$&(100,18,8,2)&1
\end{tabular}
\\
\hline
\end{tabular}
\caption{Primitive SRGs with $n$ even and $k\leq 18$, $k<n/2$}\label{SRG}
\end{table}
The graph with parameter set $a$ is the Petersen graph, which is class~2.
The complement of the Petersen graph is the triangular graph $T(5)$ which is class~1 by Alspach's result~\cite{A}.
Also Case~$h$ and one of the graphs of Case~$e$ is a triangular graph and therefore Case~1.
For the parameter sets $f$, $m$ and $o$ there is a unique SRG, the so called Lattice graph.
This SRG belongs to the Latin square graphs, and by Corollary~\ref{ls} the graph is class~1, and so is its complement.
Case~$l$ is the Gewirtz graph, which is class~1 by Lemma~\ref{ck}, as we saw in Section~\ref{suf}.
All other graphs are tested by computer (we actually tested all graphs in Table~\ref{SRG} and their complements).
Using SageMath \cite{sage},
we wrote a computer program that searches for an edge coloring in a $k$-regular graph with $k$ colors.
In each step we look (randomly) for a perfect matching, remove all its edges and continue until the remaining graph has no perfect matching.
If there are still edges left we start again.
We run this algorithm repeatedly until an edge coloring is found.
The code for this project is made freely available in a public github repository which can be found at \cite{gitcode}.
By use of this approach we found a $1$-factorization in all graphs of Table~\ref{SRG},
and in their complements, except for the Petersen graph.
Thus we found:
\begin{theorem}
With the single exception of the Petersen graph, a primitive SRG of even order and degree at most 18 is class~1 and so is its complement.
\end{theorem}
For the description of the graphs we used the website of Spence~\cite{S}.
This website also contains several SRGs with parameters $(50,21,8,9)$.
We also ran the search for these graphs. All are class~1.

It is surprising that in all cases our straightforward heuristic finds a $1$-factorization.
The heuristic is fast.
It took about one hour to find a $1$-factorization in each of the 32548 SRGs with parameter set $i$.

\end{document}